\documentclass[12pt]{amsart}

\usepackage{amsthm}
\usepackage{amssymb}
\usepackage[dvipsnames]{xcolor}
\usepackage{tikz-cd}
\usepackage{hyperref}
\hypersetup{colorlinks = true,
	linkcolor = MidnightBlue,
	urlcolor = BrickRed,
	citecolor = MidnightBlue}

\title[Chow groups of biquaternion SB-varieties]{On the Chow groups of a biquaternion Severi--Brauer variety}
\author{Eoin Mackall}
\email{eoinmackall \emph{at} gmail.com}
\urladdr{\url{www.eoinmackall.com}}
\date{\today}

\newtheorem{thm}{Theorem}[section]

\newtheorem{prop}[thm]{Proposition}
\newtheorem{cor}[thm]{Corollary}
\newtheorem{lem}[thm]{Lemma}

\theoremstyle{definition}

\newtheorem{exmp}[thm]{Example}
\newtheorem{rmk}[thm]{Remark}

\newcounter{item}

\newcommand{\HH}{\mathrm{H}}

\newcommand{\CH}{\mathrm{CH}}
\newcommand{\CT}{\mathrm{CT}}
\newcommand{\Q}{\mathrm{Q}}
\newcommand{\SB}{\mathbf{SB}}
\newcommand{\Hilb}{\mathbf{Hilb}}

\begin{document}
	\begin{abstract}
	 We provide an alternative proof that the Chow group of $1$-cycles on a Severi--Brauer variety associated to a biquaternion division algebra is torsion-free. There are three proofs of this result in the literature, all of which are due to Karpenko and rely on a clever use of $K$-theory. The proof that we give here, by contrast, is geometric and uses degenerations of quartic elliptic normal curves.
	\end{abstract}
	\maketitle
	
\section{Introduction}

That Chow groups of $1$-cycles on a Severi--Brauer variety associated to a biquaternion division algebra are torsion-free was first observed by Karpenko \cite{MR1386650}. More accurately, this statement is just one corollary of Karpenko's analysis of the topological filtration on the Grothendieck ring of a Severi--Brauer variety associated to a decomposable central simple algebra whose index and exponent differ by a squarefree factor. There Karpenko showed that the entire graded ring associated to the topological filtration was torsion-free for these varieties; the claim that these Chow groups are torsion-free then follows from the Grothendieck-Riemann-Roch without denominators comparing the Chow ring with this graded ring \cite[Example 15.3.6]{MR1644323}.

The computation of these Chow groups, showing in particular that they are torsion-free, has appeared as a consequence of more general results twice since Karpenko's work in \cite{MR1386650}. It appeared next as a consequence of the results of \cite{MR1615533}, which generalize and simplify the theorems of \cite{MR1327303}, on the topological filtration of the Grothendieck ring of a Severi--Brauer variety associated to a central simple algebra with certain 2-primary reduced behavior. It also falls out of Karpenko's recent computation of the Chow ring of a generic Severi--Brauer variety associated to a central simple algebra of index 4 and exponent 2 given in \cite{MR3581317}; see Remark \ref{rmk: simple} below for a short proof along these lines.

Here we give another proof of the fact that the Chow group of $1$-cycles on a Severi--Brauer variety associated with a biquaternion algebra are torsion-free, see Theorem \ref{thm: kar}. The proof goes as follows. First, we identify a collection of generators for these Chow groups using some results from \cite{MR4280495}; this step requires the Grothendieck-Riemann-Roch without denominators theorem \cite[Example 15.3.6]{MR1644323}. Second, we construct some explicit cycles that represent the classes of these generators; one of these cycles is characterized as having geometrically linear components, it is a geometrically a union of two pairs of skew lines forming a 4-gon of lines in 3-space, and the other is geometrically an elliptic normal curve. Finally, we construct relations between these cycles by showing that we can degenerate one into the other.

The main impetus for \textit{re}-proving this theorem was to develop new techniques for studying particular Severi--Brauer varieties. Karpenko's methods, despite having been applied very successfully in the study of generic Severi--Brauer varieties, are limited by the fact that they don't apply to arbitrary cycles. The proof that we give here is a step in this direction: we show how one can determine relations between explicit cycles using only the geometry of a Severi--Brauer variety. The more difficult problem of finding an explicit collection of cycles that are both usable and which generate the Chow group of a Severi--Brauer variety (in any positive dimension and of codimension two or more) -- without appealing to the Grothendieck-Riemann-Roch without denominators theorem -- is still open.\\

\noindent\textbf{Acknowledgments}. I want to thank Danny Krashen for explaining to me the proof of Proposition \ref{prop: rtriv}. I'd also like to thank an anonymous referee for both a careful reading of this text and for the proof contained in Remark \ref{rmk: rmkref}. Any errors contained in these proofs are now my own.

\section{Preliminaries}\label{sec: pre}
Throughout this text we fix an arbitrary field $k$ that we will use as a base. We let $A$ be an arbitrary central simple $k$-algebra and we write $X=\SB(A)$ for the associated Severi--Brauer variety, i.e.\ $X$ is the subvariety of $\mathbf{Gr}(\mathrm{deg}(A),A)$ whose $R$-points, for any finite type $k$-algebra $R$, are those projective $R$-module summands of $A\otimes_k R$ which have rank $\mathrm{deg}(A)$ and are also right ideals of $A\otimes_k R$.

We write $\zeta_X$ for the tautological sheaf on $X$, i.e.\ for the restriction of the universal subsheaf on $\mathbf{Gr}(\mathrm{deg}(A),A)$ to $X$. Given an integer $i\geq 0$, and a simple left $A^{\otimes i}$-module $M_i$, we write $\zeta_X(i)$ for the tensor product $\zeta_X^{\otimes i}\otimes_{A^{\otimes i}} M_i$; note that $M_i$ is unique up to isomorphism so that $\zeta_X(i)$ is as well. We write $\CT(X)\subset \CH(X)$ for the subring of the integral Chow ring generated by all Chern classes of $\zeta_X(1)$; it is canonically graded with summands $\CT^i(X)$ and $\CT_i(X)$ contained in $\CH^i(X)$ and $\CH_i(X)$ respectively. The following result gives a complete description of the ring structure of $\CT(X)$. (We recall that the degree of the class of a cycle $\xi$ in the Chow ring $\CH(X)$ of a Severi--Brauer variety $X$ is defined to be the degree of the pullback $\xi_{k^{alg}}$ considered in $\CH(X_{k^{alg}})$ for an algebraic closure $k^{alg}$ of the base field $k$.) 

\begin{thm}[{\cite[Proposition 3.3]{MR3581317}}]\label{thm: karct}
Assume that $A$ is a division algebra which has index $\mathrm{ind}(A)=p^n$ for some prime $p$ and any $n\geq 0$. Then $\CT^i(X)=\mathbb{Z}$ for any $0\leq i \leq \mathrm{deg}(A)-1$ with generators \[c_i(\zeta_X(1)) \quad \mbox{and} \quad c_1(\zeta_X(1))^i\]
having degrees $\mathrm{deg}(c_i(\zeta_X(1)))=\binom{p^n}{i}$ and $\mathrm{deg}(c_1(\zeta_X(1))^i)=p^{ni}$.
\end{thm}

\begin{rmk}
The proof of Theorem \ref{thm: karct} uses a specialization argument from the generic case; see the beginning of Section \ref{sec: spec} for more details. We remark that the proof doesn't require any significant input from the $K$-theory of Severi--Brauer varieties; it uses only some computations of equivariant intersection theory, see also \cite[\S8.1]{MR2258262}.
\end{rmk}

The piece of the Chow group $\CH^i(X)$ that is typically too difficult to compute is the torsion subgroup $\mathrm{Tor}_1(\CH^i(X),\mathbb{Q}/\mathbb{Z})$. But, one can try to analyze this torsion subgroup via the short exact sequence \[ 0\rightarrow \CT^i(X)\rightarrow \CH^i(X)\rightarrow \Q^i(X)\rightarrow 0.\] Here $\Q^i(X)$ is the cokernel of the canonical inclusion $\CT^i(X)\subset \CH^i(X)$ and, since the groups $\CT^i(X)$ and $\CH^i(X)$ have the same $\mathbb{Q}$-rank, there is an inclusion $\mathrm{Tor}_1(\CH^i(X),\mathbb{Q}/\mathbb{Z})\subset \Q^i(X)$. This doesn't necessarily make the problem of computing the torsion subgroup of $\CH^i(X)$ any easier. However, one can write out an explicit generating set for $\Q^2(X)$, see \cite[Proposition 3.7]{MR4280495}. As a particular case, we have:

\begin{prop}\label{prop: gen}
Suppose that $A$ is a division algebra with $\mathrm{ind}(A)=4$ and $\mathrm{exp}(A)=2$. Then $\Q^2(X)$ is generated by $c_1(\zeta_X(2))^2$.
\end{prop}

\begin{rmk}
Proposition \ref{prop: gen} can also be proved immediately by using the Grothendieck-Riemann-Roch (GRR) without denominators \cite[Example 15.3.6]{MR1644323} and Quillen's computation \cite[Theorem 4.1]{MR0338129} of the group $K(X)$ for the Severi--Brauer variety $X$. In this case, the group $K(X)$ is additively generated by the classes of $\mathcal{O}_X$, $\zeta_X(1)$, $\zeta_X(2)$, and $\zeta_X(3)$ so the GRR theorem implies $\CH^2(X)$ is additively generated by polynomials in the Chern classes of these bundles. Isolating monomials in these Chern classes shows that $\CH^2(X)$ is generated by 
\[c_1(\zeta_X(1))c_1(\zeta_X(2)), \quad c_1(\zeta_X(1))c_1(\zeta_X(3)), \quad c_1(\zeta_X(2))c_1(\zeta_X(3)),\]
\[ c_1(\zeta_X(1))^2,\quad c_1(\zeta_X(2))^2,\quad c_1(\zeta_X(3))^2, \quad c_2(\zeta_X(1)),\quad \mbox{and}\quad c_2(\zeta_X(3)).\]

Now since $A$ has exponent $\mathrm{exp}(A)=2$, the Picard group $\mathrm{Pic}(X)\cong \mathbb{Z}$ is generated by $c_1(\zeta_X(2))$; this allows one to eliminate the entire first row along with the first and third term from the second row in the generators above. Finally, there's a canonical isomorphism of bundles $\zeta_X(3)\cong \zeta_X(1)\otimes \zeta_X(2)$ so that \[c_2(\zeta_X(3))=c_2(\zeta_X(1))+3c_1(\zeta_X(2))c_1(\zeta_X(3))+6c_1(\zeta_X(2))^2\] which shows that $\CH^2(X)$ is generated by $c_1(\zeta_X(2))^2$ and $c_2(\zeta_X(1))$.
\end{rmk}

In this same case, we can say a bit more about the class $c_1(\zeta_X(2))^2$.

\begin{lem}\label{lem: ell}
Suppose that $A$ is a division algebra with $\mathrm{ind}(A)=4$ and $\mathrm{exp}(A)=2$. Then there is a smooth and irreducible curve $E\subset X$ having the following properties:
\begin{enumerate}
\item $E$ is geometrically the intersection of two quadric surfaces;
\item we have $[E]=c_1(\zeta_X(2))^2$ inside $\CH_1(X)$.
\end{enumerate}
\end{lem}

\begin{proof}
By Albert's theorem, the algebra $A$ decomposes $A\cong Q_1\otimes Q_2$ as a tensor product of two quaternion algebras $Q_1,Q_2$ (for an algebraic proof, see \cite[Theorem 16.1]{MR1632779}; for a proof by geometric methods, see \cite[Theorem 5.5]{MR657430}). Set $Y_1=\SB(Q_1)$ and $Y_2=\SB(Q_2)$. There is an embedding \[s:Y_1\times Y_2\rightarrow X\] which is geometrically, i.e.\ over an algebraic closure $k^{alg}$ of the base field, isomorphic to the Segre embedding of a quadric surface in $\mathbb{P}^3$.

Let $E\subset Y_1\times Y_2$ be a general hyperplane section from the complete linear system associated to the line bundle $(\zeta_{Y_1}(2)\boxtimes \zeta_{Y_2}(2))^\vee$. We can find such an $E$ that is smooth and geometrically irreducible by Bertini's theorem \cite[Th\'eor\`eme 6.10 et Corollaire 6.11]{MR725671} using that the base field is infinite (since all division algebras over finite fields are trivial).

We have that $s_*([Y_1\times Y_2])=-c_1(\zeta_X(2))$ using that $\mathrm{deg}(Y_1\times Y_2)=2$ in $X$ and $\mathrm{Pic}(X)=\mathbb{Z}$ with a generator of degree $2$. 
One can also check (over $k^{alg}$) that $s^*\zeta_X(1)\cong \zeta_{Y_1}(1)\boxtimes \zeta_{Y_2}(1)$. Hence, it follows \begin{align*}c_1(\zeta_X(2))^2 & =-s_*([Y_1\times Y_2])\cdot c_1(\zeta_X(2))\\ &= -s_*([Y_1\times Y_2]\cdot c_1(\zeta_{Y_1}(2)\boxtimes \zeta_{Y_2}(2)))\\ &= -s_*(c_1(\zeta_{Y_1}(2)\boxtimes \zeta_{Y_2}(2))) \\ &=-s_*(-[E])=s_*([E])=[E].
\end{align*} 
This proves part (2) of the lemma and (1) follows from the construction. Note that $E$ has genus $g(E)=1$ and degree $\mathrm{deg}(E)=4$ inside $X$. The former of these can be checked over $k^{alg}$ as $E_{k^{alg}}$ is a curve of bidegree $(2,2)$ on a quadric surface.
\end{proof}

\section{The Proof}
In this section, we prove the main theorem. We leave the verification of some specific details until later sections.

\begin{thm}[{\cite[Proposition 5.1]{MR1615533}}]\label{thm: kar}
Let $X=\SB(Q_1\otimes Q_2)$ be the Severi--Brauer variety associated to a biquaternion division algebra. Then $\CH_1(X)=\mathbb{Z}$ is torsion-free.
\end{thm}

\begin{proof}
To follow the notation of Section \ref{sec: pre}, we write $A=Q_1\otimes Q_2$. We're going to verify explicitly the equality of cycle classes \[c_1(\zeta_X(2))^2=2(3c_2(\zeta_X(1))-c_1(\zeta_X(1))^2)\] which will show that $\Q^2(X)=0$ by Proposition \ref{prop: gen}.

Let $F/k$ be a biquadratic Galois splitting field for $A$ with Galois group $G=\mathrm{Gal}(F/k)$. Choose generators $\sigma,\tau$ so that $G=\{1,\sigma,\tau,\sigma\tau\}$. Find a closed point $x\in X$ with residue field $k(x)=F$ and identify the $F$-points of $x_F$ with elements of $G$. Let $L_{\sigma,\sigma\tau}\subset X_F$ denote the line passing through $\sigma$ and $\sigma\tau$. We define the lines $L_{1,\sigma}$, $L_{\tau,\sigma\tau}$, and $L_{1,\tau}$ similarly. Now each of the unions \[ C'=L_{1,\sigma}\cup L_{\tau,\sigma\tau} \quad \mbox{and}\quad D'=L_{1,\tau}\cup L_{\sigma,\sigma\tau}\] form a $G$-orbit and hence descend to curves $C,D\subset X$ with $C\times_k F=C'$ and $D\times_k F=D'$.

By Corollary \ref{cor: curvescd} below, both curves $C$ and $D$ represent the cycle class $3c_2(\zeta_X(1))-c_1(\zeta_X(1))^2$ so that \[[C\cup D] = [C]+[D]=2(3c_2(\zeta_X(1))-c_1(\zeta_X(1))^2).\] We're going to show that $C\cup D$ is a rational degeneration of any curve $E$ from Lemma \ref{lem: ell} so that they represent the same cycle class. More precisely, we're going to show there exists a proper surface $S$ with a flat morphism to $\mathbb{P}^1$, say $\pi:S\rightarrow \mathbb{P}^1$, and a morphism $\rho:S\rightarrow X$ so that $\rho$ is a closed immersion on any fiber of $\pi$ and the restriction of $\rho$ to the fibers over points $t_0,t_1\in \mathbb{P}^1(k)$ are $\rho(\pi^{-1}(t_0))=E$ and $\rho(\pi^{-1}(t_1))=C\cup D$. Since $\pi$ is flat in this scenario, and since $[t_0]=[t_1]\in \CH_0(\mathbb{P}^1)$, we find an equality \[[E]=\rho_*\pi^*(t_0)=\rho_*\pi^*(t_1)=[C\cup D]\] in $\CH_1(X)$ which implies that $\Q^2(X)=0$ as desired.

It remains to construct the triple $(S,\pi,\rho)$. For this, we first consider the Hilbert scheme $\Hilb_{4t}(\mathbb{P}^3_F/F)$ parametrizing $F$-subschemes of $\mathbb{P}^3_F$ having Hilbert polynomial $h(t)=4t$. There is a rational map \[\Lambda:\mathbf{Gr}(2,\Gamma(\mathbb{P}^3_F,\mathcal{O}(2))) \dashrightarrow \Hilb_{4t}(\mathbb{P}^3_F/F)\] sending a $2$-dimensional subspace $V\subset \Gamma(\mathbb{P}^3_F,\mathcal{O}(2))$ to the vanishing set of those polynomials in $V$. Writing $\mathbb{P}=\mathbb{P}(\Gamma(\mathbb{P}^3_F,\mathcal{O}(1)))$, the map $\Lambda$ is defined outside the image of the multiplication map \[\mathbb{P}\times \mathbf{Gr}(2,\Gamma(\mathbb{P}^3_F,\mathcal{O}(1)))\rightarrow \mathbf{Gr}(2,\Gamma(\mathbb{P}^3_F,\mathcal{O}(2))),\quad (f,\langle g,h\rangle) \mapsto \langle fg,fh\rangle\] and the image of $\Lambda$ is contained inside $H_{4,1,3}$, the irreducible component of $\Hilb_{4t}(\mathbb{P}^3_F/F)$ containing the open subscheme parametrizing smooth curves of degree 4 and genus 1. Since $H_{4,1,3}$ is geometrically irreducible and of dimension $16$ by \cite[Theorem 8]{MR875083}, we find that $\Lambda$ is birational between its domain and this component $H_{4,1,3}$.

Let $\xi\in Z^1(G,\mathrm{PGL}_4(F))$ be a Galois $1$-cocycle representing the class of $X$ inside of $\HH^1(G,\mathrm{PGL}_4(F))$. Using $\xi$ one can descend $\Lambda$ to a map \[\Lambda_{\xi}: \SB_2(S^2(A))\dashrightarrow \Hilb_{4t}^{tw}(X/k)\] from a generalized Severi--Brauer variety of the second symmetric power $S^2(A)$ of $A$ to the $\xi$-twisted Hilbert scheme of $X$ (for more on twisted Hilbert schemes see Section \ref{sec: twhilb}). The map $\Lambda_\xi$ is also a birational map between its domain and the irreducible component $\mathrm{Ell}_4(X)$, containing the image of $\Lambda_\xi$, which is an $F/k$-form of $H_{4,1,3}$.

The algebra $S^2(A)$ is split by \cite[Example 4.5]{MR657430} and in this case $\SB_2(S^2(A))=\mathbf{Gr}(2,\Gamma(X,\zeta_X(2)^\vee))$. Lemmas \ref{lem: degen} and \ref{lem: ellh} below show that the $k$-points corresponding to $C\cup D$ and $E$ lie in the image of $\Lambda_\xi$. Hence there is a rational map \[\mathbb{P}^1\dashrightarrow \mathbf{Gr}(2,\Gamma(X,\zeta_X(2)^\vee))\dashrightarrow \Hilb_{4t}^{tw}(X/k)\] extending to a genuine morphism $\phi:\mathbb{P}^1\rightarrow\Hilb_{4t}^{tw}(X/k)$ that passes through these two points. 

The universal family $\mathbf{Univ}_{4t}(\mathbb{P}^3_F/F)$ on the scheme $\Hilb_{4t}(\mathbb{P}^3_F/F)$ descends to a universal family $\mathbf{Univ}_{4t}^{tw}(X/k)$ on the $\xi$-twisted Hilbert scheme $\Hilb_{4t}^{tw}(X/k)$. Letting $S$ be the surface $\mathbf{Univ}_{4t}^{tw}(X/k)\times_{\phi} \mathbb{P}^1$, considered as a closed subscheme of $X\times_k \mathbb{P}^1$, and denoting by $\pi,\rho$ the corresponding projections, it follows that $(S,\pi,\rho)$ satisfy all the desired properties, completing the proof.
\end{proof}

\begin{rmk}
The existence of the map $\Lambda_\xi$ does not depend on the assumption that $A$ has exponent $\exp(A)=2$; it exists more generally for any division algebra $A$ with $\mathrm{ind}(A)=4$. In the case $\mathrm{exp}(A)=4$, the algebra $S^2(A)$ has index $\mathrm{ind}(S^2(A))=2$ by \cite[Example 4.5]{MR657430} and so $\SB_2(S^2(A))$ is also rational by \cite[Proposition 3]{MR1092553}. Hence the open subscheme of $\mathrm{Ell}_4(X)$ parametrizing smooth and irreducible curves of genus 1 has a $k$-rational point. This gives an alternative proof of a result from \cite{MR3091612}; namely, that every Severi--Brauer variety associated to a central simple algebra of index $4$ contains a smooth curve of genus 1. 
\end{rmk}

\begin{rmk}\label{rmk: rmkref}
	There are two main parts to the proof of Theorem \ref{thm: kar}. The first part, completed in Corollary \ref{cor: curvescd}, consists of proving that both of the curves $C$ and $D$ represent the cycle class $3c_2(\zeta_X(1))-c_1(\zeta_X(1))^2$. The remainder of this paper is devoted to proving this representation: in Section \ref{sec: spec} we show, by specialization from the generic case, that this linear combination of Chern classes is represented by a subscheme of $X$ which is geometrically a union of two skew lines; we then show, in Section \ref{sec: twhilb}, that every subscheme of $X$ which is geometrically a union of two skew lines represents the same cycle class in $\CH_1(X)$ through a study of the connectedness properties of the scheme which parametrizes all such subschemes of $X$.
	
	The second main part of the proof of Theorem \ref{thm: kar} is the proof that both of the classes $[C\cup D]$ and $[E]$, for a geometrically elliptic normal curve $E\subset X$, are equal inside $\CH_1(X)$. In the above proof, we show that $[C\cup D]=[E]$ by constructing a rational degeneration from one subscheme to the other. We can give another proof of this equality without the use of twisted Hilbert schemes (although, twisted Hilbert schemes are still used in Section \ref{sec: twhilb} below for the first part of the proof); in the following, we show $[E]=[C\cup D]$ using Hilbert's Theorem 90. I'd like to thank an anonymous referee for explaining to me this proof.
	
	Keep notation as in the proof of Theorem \ref{thm: kar}. Let $W\subset \Gamma(X,\zeta_X(2)^\vee)$ be the $k$-vector subspace of sections whose vanishing locus contains both $C$ and $D$. Over the biquadratic splitting field $F/k$, the $F$-vector space $W_F\subset \Gamma(\mathbb{P}^3_F,\mathcal{O}(2))$ consists of all sections with vanishing locus containing both $C'$ and $D'$. One can check directly that $\mathrm{dim}_F(W_F)\geq 2$ (it's possible to construct two different unions of two planes which both vanish at $C'$ and $D'$ using the $F$-points of $x_F$ where $x\in X$ has $k(x)\cong F$; see the proof of Lemma \ref{lem: degen}). So, by Hilbert's Theorem  90 (essentially a dimension count in this case), we find that $\mathrm{dim}_k(W)\geq 2$.
	
	If $\theta$ and $\eta$ are two linearly independent elements of $W$, then the zero sections of $\theta$ and $\eta$ both represent $c_1(\zeta_X(2)^\vee)$. Also, the intersection of the zero sections of $\theta$ and $\eta$ is exactly $C\cup D$. So there is an equality \[[C\cup D]=c_1(\zeta_X(2)^\vee)^2=c_1(\zeta_X(2))^2=[E]\] inside $\CH_1(X)$ by \cite[\href{https://stacks.math.columbia.edu/tag/0B1I}{Tag 0B1I}]{stacks-project} and Lemma \ref{lem: ell}.
\end{rmk}

\begin{rmk}\label{rmk: simple}
A short proof of Theorem \ref{thm: kar} can be found in Karpenko's computation of the Chow ring of a Severi--Brauer variety associated to a generic central simple algebra of index 4 and exponent 2 \cite{MR3581317}. Namely, there's an equality of classes inside the $\lambda$-ring $K(X)$ \[\lambda^2([\zeta_X(1)])=6[\zeta_X(2)]\quad \mbox{and} \quad  [\zeta_X(1)]^2=16[\zeta_X(2)]\] so that $3\lambda^2([\zeta_X(1)])-[\zeta_X(1)]^2=2[\zeta_X(2)]$ is contained in the $\lambda$-subring of $K(X)$ generated by $\zeta_X(1)$. It follows that Chern classes of $\zeta_X(2)^{\oplus 2}$ are contained in $\CT(X)$. So $c_2(\zeta_X(2)^{\oplus 2})=c_1(\zeta_X(2))^2$ is also contained in $\CT(X)$ and therefore $Q^2(X)=0$.
\end{rmk}

\section{Specialization Arguments}\label{sec: spec}
In this section, we prove some results by specialization from the generic case. We work over a fixed base field $k$ and now we let $A$ be a given central division $k$-algebra of degree $\mathrm{deg}(A)=n$ with $X=\SB(A)$. We choose an embedding $G=\mathrm{PGL}_n\subset \mathrm{GL}_N$ of algebraic groups and consider the quotient $S=\mathrm{GL}_N/G$. The quotient map $\mathrm{GL}_N\rightarrow S$ is a $G$-torsor and, if $P\subset G$ is a parabolic subgroup with $G/P\cong \mathbb{P}^{n-1}$, the quotient $\mathrm{GL}_N/P\rightarrow S$ is a Severi--Brauer $S$-scheme. 

From now on we write $\mathcal{X}=\mathrm{GL}_N/P$ for this Severi--Brauer $S$-scheme and $X^{gen}$ for its generic fiber over $S$. The Severi--Brauer variety $X^{gen}$ is associated to a division $k(S)$-algebra $U^{gen}$ having index $\mathrm{ind}(U^{gen})=n$ and exponent $\mathrm{exp}(U^{gen})=n$. Moreover, $X^{gen}$ is versal in the sense that for any Severi--Brauer variety $Y$ over a field extension $F/k$ with $F$ infinite, there is an $F$-point $s\in S(F)$ so that the fiber $\mathcal{X}_s$ is isomorphic with $Y$. Since $S$ is smooth we can find, at any point $s\in S$, a sequence of DVRs say $(R_0,\mathfrak{m}_0),...,(R_{j(s)},\mathfrak{m}_{j(s)})$ satisfying the following conditions:
\begin{enumerate}
	\item $\mathrm{Frac}(R_0)=k(S)$,  
	\item $R_i/\mathfrak{m}_i\cong \mathrm{Frac}(R_{i+1})$
	\item $R_{j(s)}/\mathfrak{m}_{j(s)}\cong k(s)$.
\end{enumerate}
Hence, if $Y$ is defined over the base field $k$, this means that there exist specialization homomorphisms \cite[\S20.3]{MR1644323} \[\CH^i(X^{gen})\rightarrow \CH^i(Y)\] which take $c_i(\zeta_{X^{gen}})$ to $c_i(\zeta_{Y})$ or, in the case that $Y$ is associated to a division algebra of index $n$, that take $c_i(\zeta_{X^{gen}}(1))$ to $c_i(\zeta_Y(1))$.

The Chow groups $\CH_i(X^{gen})$ are torsion-free for all $i\geq 0$ by \cite[Proposition 3.2]{MR3590349}. For those integers $i$ where these groups are nonzero (i.e.\ for $i=0,...,n-1$) this means that $\CH_i(X^{gen})=\mathbb{Z}$. Passing to a splitting field for $X^{gen}$, one can determine that generators of $\CH_i(X^{gen})$ have degrees $\pm n/\gcd(n,n-i-1)$.

Using this description of the Chow groups of $X^{gen}$ together with the specialization maps above, we provide in Proposition \ref{prop: reps} a description of some cycle representatives for the Chern classes $c_i(\zeta_X(1))$; this is not used in the proof of Theorem \ref{thm: kar} but it is of independent interest. Then we turn to the case that $A$ is a biquaternion algebra and prove Lemma \ref{lem: spec1} below on representatives for a generator of $\CT^2(X)$; this is used explicitly in the proof of Theorem \ref{thm: kar}.

\begin{prop}\label{prop: reps}
For any integer $i\geq 1$, the class $(-1)^ic_i(\zeta_X(1))$ is represented by the cycle class of a scheme which is geometrically a union of $i$-dimensional linear spaces.
\end{prop}

\begin{proof}
Choose a field extension $F/k(S)$ of degree $n$ splitting $X^{gen}$ and let $E/F$ be a Galois closure of $F$ inside a fixed separable closure $F^{sep}$. 
Choose any point $x$ in $X^{gen}$ with residue field $F$. The set of $E$-points in $x_E$ has exactly $n$ elements which form an orbit for the Galois group $\mathrm{Gal}(E/F)$. For any integer $i\geq 1$, denote by $V^{gen}_{i,E}\subset X^{gen}_E$ the union of all $i$-dimensional linear spaces of $X^{gen}_E$ passing through any $(i+1)$-tuple of the $n$ points contained in $x_E$. The union $V^{gen}_{i,E}$ is Galois stable, so it descends to a subvariety $V^{gen}_i$ of $X^{gen}$.

As $X$ is a Severi--Brauer variety defined over $k$ there is some $k$-point $s\in S(k)$ with $\mathcal{X}_s\cong X$. Let $(R_0,\mathfrak{m}_0),..., (R_{j(s)},\mathfrak{m}_{j(s)})$ be a sequence of DVRs connecting the fields $k(S)$ and $k(s)$ as above. Write $\mathcal{V}_{i,0}$ for the closure of $V_i^{gen}$ inside $\mathcal{X}\times_S R_0$.

Let $T_0'$ be the integral closure of $R_0$ inside $E$ and write $T_0=(T_0')_{\mathfrak{n}_0}$ for the localization at a fixed maximal ideal $\mathfrak{n}_0\subset T'_0$ lying above $\mathfrak{m}_0$. By construction, the fraction field of $T_0$ is $E$ and the DVR $(T_0,\mathfrak{n}_0)$ is flat over $R_0$. Now the Severi--Brauer scheme $\mathcal{X}\times_S T_0$ is split since it is generically split \cite[Lemma A.1]{em}. Moreover, $\mathcal{X}\times_S T_0\cong \mathbb{P}^{n-1}_{T_0}$ since every finitely generated projective $T_0$-module is free.

The scheme $\mathcal{V}_{i,0}\times_S T_0$ is generically a union of $\binom{n}{i+1}$ linear spaces. The closure of each of these linear spaces intersects the special fiber $\mathcal{X}\times_S (T_0/\mathfrak{n}_0)$ in a linear space as well. Hence the irreducible components of $\mathcal{V}_{i,0}\times_S (T_0/\mathfrak{n}_0)$ must all be linear (possibly with multiplicities).

The field extension $T_0/\mathfrak{n}_0$ of $R_0/\mathfrak{m}_0$ is finite since $R_0$ is a Nagata ring \cite[\href{https://stacks.math.columbia.edu/tag/0335}{Tag 0335}]{stacks-project}. Let $T_1'$ be the integral closure of $R_1$ in $T_0/\mathfrak{n}_0$ and $T_1=(T_1')_{\mathfrak{n}_1}$ the localization at a fixed maximal ideal $\mathfrak{n}_1$ lying over $\mathfrak{m}_1$. Then $T_1$ is finite over $R_1$ since $R_1$ is also a Nagata ring. Write $\mathcal{V}_{i,1}$ for the closure of $\mathcal{V}_{i,0}\times_S (R_0/\mathfrak{m}_0)$ inside $\mathcal{X}\times_S R_1$. It follows, as above, that $\mathcal{V}_{i,1}\times_S (T_1/\mathfrak{n}_1)$ is a union of linear spaces (possibly with multiplicities).

Repeating this process, we get a sequence of spaces $\mathcal{V}_{i,j}$ contained in $\mathcal{X}\times_S R_j$ with both generic and special fibers over $R_j$ the union of some linear spaces. It follows that the specialization of the cycle class defined by $V_i^{gen}$ to $X$ is represented by a scheme $\mathcal{V}_{i,j(s)}\times_S k(s)$ inside $X$ which is geometrically a union of linear spaces \cite[\S20.3]{MR1644323}. Using \cite[Proposition 3.2]{MR3590349} one can check that $V_i^{gen}$ represents the class of $(-1)^ic_i(\zeta_{X^{gen}}(1))$ by comparing degrees geometrically so, by the discussion above the proposition, this concludes the proof. \end{proof}

Using ideas from the proof of the above proposition, we can prove the following lemma that is used in the proof of Corollary \ref{cor: curvescd} below:

\begin{lem}\label{lem: spec1}
Suppose that $A=Q_1\otimes Q_2$ is a biquaternion division algebra. Then there exists a subscheme $V\subset X$ which is geometrically a union of two skew lines and such that \[[V]=3c_2(\zeta_X(1))-c_1(\zeta_X(1))^2\] inside $\CH_1(X)$ with $X=\SB(A)$.
\end{lem}

\begin{proof}
We assume for this proof that $X^{gen}$ is the generic Severi--Brauer variety associated to a generic central division $k$-algebra $U^{gen}$ of index $\mathrm{ind}(U^{gen})=4$ and exponent $\mathrm{exp}(U^{gen})=4$. By a theorem of Albert \cite[\S XI.6 Theorem 9]{MR0000595}, there is a quartic Galois field extension $F/k$ contained in $U^{gen}$ with Galois group $\mathrm{Gal}(F/k)=\mathbb{Z}/2\mathbb{Z}\times \mathbb{Z}/2\mathbb{Z}$. Let $\sigma$ and $\tau$ be two generators for this group.

Pick a point $x$ of $X^{gen}$ with residue field $k(x)=F$. Identify the $F$-points of $x_F$ with the set $\{1,\sigma,\tau,\sigma\tau\}$ in a way that is compatible with the Galois action on $X^{gen}_F$. Then the union \[ V'=L_{1,\sigma}\cup L_{\tau,\sigma\tau},\] of the lines passing through $\{1,\sigma\}$ and $\{\tau,\sigma\tau\}$ respectively, is Galois stable. Hence there is a subscheme $V^{gen}\subset X^{gen}$ with $V^{gen}_F\cong V'$.

Let $s\in S(k)$ be such that $\mathcal{X}_s\cong X$, and fix a sequence of DVRs $(R_0,\mathfrak{m}_0),..., (R_{j(s)},\mathfrak{m}_{j(s)})$ connecting the fields $k(S)$ and $k(s)$ as before. Following the argument of Proposition \ref{prop: reps}, we arrive at a subscheme $V\subset X$ which is geometrically a union of lines. Again, by using \cite[Proposition 3.2]{MR3590349} and comparing degrees, one can check that the cycle class of $V$ represents the class $3c_2(\zeta_X(1))-c_1(\zeta_X(1))^2$. 

There are three possible cases: either $V$ is geometrically the union of two skew lines, $V$ is geometrically the union of two lines contained inside some plane, or $V$ is geometrically a double line. However, since $X$ is associated to a division algebra, the subscheme $V$ is geometrically nondegenerate \cite[Lemma 3.4]{em} leaving only the possibility that $V$ is geometrically the union of two skew lines.
\end{proof}

\section{Twisted Hilbert Schemes}\label{sec: twhilb}
Twisted Hilbert schemes are constructed in \cite{em}. In the setting that we're interested in, we can describe these schemes as twisted forms of Hilbert schemes along a Galois $1$-cocycle. We keep notation as above: $k$ is a given base field, $F/k$ is a finite Galois field extension with Galois group $G=\mathrm{Gal}(F/k)$, and $\xi\in Z^1(G,\mathrm{PGL}_{n+1}(F))$ is a $1$-cocycle.

Let $\phi(t)$ be in $\mathbb{Q}[t]$. Fix an $F$-automorphism $\alpha$ in $\mathrm{Aut}_F(\mathbb{P}^n\times_k F)$. Then $\alpha$ induces an automorphism of the scheme $\Hilb_{\phi(t)}(\mathbb{P}^n\times_k F/F)$ due to the representability of the Hilbert scheme. In more detail, the data of a morphism $V\rightarrow \Hilb_{\phi(t)}(\mathbb{P}^n\times_k F/F)$ is precisely the data of a subscheme $U\subset \mathbb{P}^n\times V$ that is flat and proper over $V$ with Hilbert polynomial $\phi(t)$. By representability, there is a universal subscheme \[\mathbf{Univ}_{\phi(t)}(\mathbb{P}^n\times_k F/F)\subset \mathbb{P}^n\times \Hilb_{\phi(t)}(\mathbb{P}^n\times_k F/F)\]
and $\alpha$ induces the automorphism corresponding to the subscheme \[(\alpha \times \mathrm{Id})(\mathbf{Univ}_{\phi(t)}(\mathbb{P}^n\times_k F/F))\subset \mathbb{P}^n\times \Hilb_{\phi(t)}(\mathbb{P}^n\times_k F/F)\] gotten by composition with $\alpha$.

Identifying $\mathrm{PGL}_{n+1}(F)$ with $\mathrm{Aut}_F(\mathbb{P}^n\times_k F)$, we get a $1$-cocycle $\xi'$ for $G$ with values in $\mathrm{Aut}_F(\Hilb_{\phi(t)}(\mathbb{P}^n\times_k F/F))$ by pushing forward $\xi$. If $X$ is the Severi--Brauer variety that one gets by twisting $\mathbb{P}^n\times_k F$ by $\xi$, then the $\xi$-twisted Hilbert scheme $\Hilb_{\phi(t)}^{tw}(X/k)$ is the scheme that one gets by twisting $\Hilb_{\phi(t)}(\mathbb{P}^n/k)\times_k F$ by $\xi'$. The twisted Hilbert scheme satisfies a representability property similar to the usual Hilbert scheme \cite[Theorem 2.5]{em}.

\begin{rmk}
Let $\mathcal{L}$ be a very ample line bundle on $X$ of degree $m$. Then the twisted Hilbert schemes associated to $X$ are related to the usual Hilbert schemes of $X$, constructed with respect to the complete linear system associated with $\mathcal{L}$, by an isomorphism \[\Hilb^{tw}_{\phi(t)}(X/k)\cong \Hilb_{\phi(mt)}(X/k).\] The benefit to considering the twisted version of these schemes, rather than their untwisted counterparts, is that the twisted versions allow one to avoid the choice of a projective embedding of $X$.
\end{rmk}

In the proof of Theorem \ref{thm: kar}, we used the following result:

\begin{lem}\label{lem: degen}
Let $A$ be a division $k$-algebra and assume $\mathrm{ind}(A)=4$. For a Galois splitting field $F/k$ of $A$, define the rational map $\Lambda$ \[\Lambda:\mathbf{Gr}(2,\Gamma(\mathbb{P}^3_F,\mathcal{O}(2))) \dashrightarrow \Hilb_{4t}(\mathbb{P}^3_F/F)\] by sending a $2$-dimensional subspace $V\subset \Gamma(\mathbb{P}^3_F,\mathcal{O}(2))$ to the vanishing set of those polynomials in $V$ and let \[\Lambda_\xi:\SB_2(S^2(A))\dashrightarrow \Hilb_{4t}^{tw}(X/k)\] be the map that one gets by twisting.

Now assume that $A=Q_1\otimes Q_2$ decomposes into a tensor product of two quaternion algebras and fix a biquadratic Galois extension $F/k$ splitting $A$ and with Galois group $G=\{1,\sigma,\tau,\sigma\tau\}$. Pick a point $x\in X$ with residue field $F$ and label the $F$-points of $x_F$ with elements of $G$. Let $C,D\subset X$ be those curves with $C'=C\times_k F$ and $D'=D\times_k F$ the unions \[C'=L_{1,\sigma}\cup L_{\tau,\sigma\tau} \quad \mbox{and}\quad D'=L_{1,\tau}\cup L_{\sigma,\sigma\tau}\] where $L_{p,q}$ is the line through the points $p$ and $q$.

Then the point $[C\cup D]$ of $\Hilb_{4t}^{tw}(X/k)(k)$ is in the image of $\Lambda_\xi$.
\end{lem}

\begin{proof}
The map $\Lambda$ is birational between its domain and an irreducible component of its target. Since this implies the same for $\Lambda_\xi$, it suffices to prove that there is some algebraic field $K/k$ so that $[C_K\cup D_K]$ is in the image of $\Lambda_{\xi}\times_k K$. We'll show that this holds when $K=F$.

Let $P(1,\sigma,\sigma\tau)$ be the plane passing through the points $1,\sigma$ and $\tau$. We note that this makes sense because no three of the points of $x_F$ lie on a line (the $F$-points of $x_F$ span all of $\mathbb{P}^3_F$). Define planes $P(1,\tau,\sigma\tau)$, $P(1,\sigma,\tau)$, and $P(\sigma,\tau,\sigma\tau)$ similarly.

Now we have \[\left(P(1,\sigma,\sigma\tau)\cup P(1,\tau,\sigma\tau)\right) \cap \left(P(1,\sigma,\tau)\cup P(\sigma,\tau,\sigma\tau)\right)=C_F\cup D_F\]
with the left-hand side the vanishing set of two (independent) quadratic polynomials, i.e.\ defining an $F$-point in $\mathbf{Gr}(2,\Gamma({\mathbb{P}^3_F,\mathcal{O}(2)}))$.
\end{proof}

\begin{exmp}
In the above proof, neither of the unions \[P(1,\sigma,\sigma\tau)\cup P(1,\tau,\sigma\tau)\quad \mbox{and}\quad  P(1,\sigma,\tau)\cup P(\sigma,\tau,\sigma\tau)\] define Galois stable sections descending to elements in $\Gamma(X,\zeta_X(2)^\vee)$. However, there are always two linearly independent sections of $\zeta_X(2)^\vee$ with vanishing locus the scheme $C\cup D$. Here's an example.

Let $\mathrm{char}(L)\neq 2$, let $k=L(a,b,c,d)$ for four independent variables $a,b,c,d$ and let $F=k(\sqrt{a},\sqrt{b})$. Let $G=\mathrm{Gal}(F/k)$ be the Galois group with elements $\{1,\sigma,\tau,\sigma\tau\}$ acting on $F$ by \[\sigma(\sqrt{a})=-\sqrt{a},\, \sigma(\sqrt{b})=\sqrt{b},\, \tau(\sqrt{a})=\sqrt{a},\, \mbox{and}\,\tau(\sqrt{b})=-\sqrt{b}.\] Take the $1$-cocycle $\xi$ with values in $\mathrm{PGL}_4(F)$ defined by the following matrices: \[\xi(1)=\begin{pmatrix} 1 & 0 & 0 & 0\\ 0 & 1 & 0 & 0 \\ 0 & 0 & 1 & 0\\ 0 & 0 & 0 & 1\end{pmatrix} \quad \xi(\sigma)=\begin{pmatrix} 0 & 0 & c & 0\\ 0 & 0 & 0 & c \\ 1 & 0 & 0 & 0\\ 0 & 1 & 0 & 0\end{pmatrix}\] \[\xi(\tau)=\begin{pmatrix} 0 & d & 0 & 0\\ 1 & 0 & 0 & 0 \\ 0 & 0 & 0 & d\\ 0 & 0 & 1 & 0\end{pmatrix} \quad \xi(\sigma\tau)=\begin{pmatrix} 0 & 0 & 0 & cd\\ 0 & 0 & c & 0 \\ 0 & d & 0 & 0\\ 1 & 0 & 0 & 0\end{pmatrix}.\]

Inside $\mathbb{P}^3_F$, the points $e_1,e_2,e_3,e_4$ defined by $e_i=(\delta_{ij})_{j=1}^4$ with $\delta_{ij}=1$ if $i=j$ and $\delta_{ij}=0$ otherwise, corresponding to standard basis vectors, form a $G$-orbit under the $\xi$-twisted action. Hence this $G$-orbit descends to a point $x$ with $k(x)=F$ on the Severi--Brauer variety $X$ obtained by twisting $\mathbb{P}^3_F$ along $\xi$.

In this setting, there are sections $\theta,\eta\in \Gamma(X,\zeta_X(2)^\vee)$ so that the vanishing loci intersect to $V(\theta\times_k F)\cap V(\eta\times_k F)=C_F\cup D_F$ over $F$. Moreover, the sections $\theta,\eta$ are determined by \[\theta\times_k F=x_2x_3+x_1x_4\quad \mbox{and} \quad \eta\times_k F= (\sqrt{ab})x_2x_3-(\sqrt{ab})x_1x_4\] with $x_1,x_2,x_3,x_4$ the usual coordinate sections on $\mathbb{P}^3_F$.
\end{exmp}

\begin{lem}\label{lem: ellh}
Keep the notation of Lemma \ref{lem: degen}. Suppose that $E\subset X$ is a smooth curve which is geometrically the intersection of two quadric surfaces. Then the point $[E]$ of $\Hilb_{4t}^{tw}(X/k)(k)$ is in the image of $\Lambda_\xi$.
\end{lem}

\begin{proof}
As in the proof of Lemma \ref{lem: degen}, it suffices to check that $[E_F]$ is in the image of $\Lambda$ and this is true by assumption.	
\end{proof}

Consider now the twisted Hilbert scheme $\Hilb_{2t+2}^{tw}(X/k)$. This space has two irreducible components. Geometrically, the general point of one of these components parametrizes two skew lines and the general point of the other parametrizes unions of a conic and an isolated point. Both of these irreducible components are smooth.

We write $\mathrm{Skew}(X)$ for the irreducible component of $\Hilb_{2t+2}^{tw}(X/k)$ which parametrizes those subschemes of $X$ that are geometrically the union of two skew lines. There is an isomorphism \[\mathrm{Skew}(X)\cong \mathrm{Bl}_{\Delta}(S^2(\SB_2(A)))\] with the blow up of the symmetric square of the second generalized Severi--Brauer variety associated to $A$ along the diagonal. To see this one can consider the Galois action over an algebraic closure and appeal to \cite[Theorem 1.1 (3)]{MR2834144}.

In particular, the scheme $\mathrm{Skew}(X)$ contains, as an open subscheme, the variety $\acute{e}t_2(A)$ constructed in \cite{MR2601007}. There the scheme $\acute{e}t_2(A)$ is constructed as follows. First, let $V=\SB_2(A)\times \SB_2(A)$ and consider the open subscheme $V_*$ parametrizing pairs of ideals $(I,J)$ such that $I\cap J=0$ and $I+J=A$. The symmetric group $S_2$ acts on $V_*$ freely, and $\acute{e}t_2(A)$ is defined as the quotient $\acute{e}t_2(A)=V_*/S_2$.

In \cite[Theorem 6.6]{MR2601007} it's shown that $\acute{e}t_2(A)$ is $R$-trivial when the characteristic of $k$ is not $2$. In fact, this result is also true if the characteristic of $k$ is $2$.

\begin{prop}\label{prop: rtriv}
Suppose that $A$ is a biquaternion division algebra. Then $\acute{e}t_2(A)$ is $R$-trivial, i.e.\ for any field extension $F/k$ and for any two points $x,y\in \acute{e}t_2(A)(F)$ there is a rational map $\phi:\mathbb{P}^1\dashrightarrow \acute{e}t_2(A)$ so that $x=\phi(F)(p)$ and $y=\phi(F)(q)$ for some $p,q\in \mathbb{P}^1(F)$.
\end{prop}

\begin{proof}
It suffices to prove the remaining case, i.e.\ assuming that the characteristic of $k$ is $2$. The proof is nearly identical to that of \cite[Theorem 6.6]{MR2601007} with only minor adjustments. The algebra $A$ is equipped with a canonical symplectic involution $\sigma$ \cite[Proposition 2.23]{MR1632779}. The set $\mathrm{Symd}(A,\sigma)=\{a+\sigma(a): a\in A\}$ of symmetrized elements of $A$ is a 6-dimensional $k$-vector space \cite[Proposition 2.6 (2)]{MR1632779}. 

For every right ideal $I\subset A$ of reduced dimension $2$, the intersection $I\cap \mathrm{Symd}(A,\sigma)$ is a $1$-dimensional subspace of $\mathrm{Symd}(A,\sigma)$ which is isotropic for the quadratic Pfaffian norm form $\mathrm{Nrp}_\sigma$ on $\mathrm{Symd}(A,\sigma)$. This assignment defines an isomorphism \cite[Proposition 15.20]{MR1632779} between $\SB_2(A)$ and the quadric $V(\mathrm{Nrp}_\sigma)\subset \mathbb{P}(\mathrm{Sym}(A,\sigma))$.

Let $W$ be the quotient $((V(\mathrm{Nrp}_\sigma)\times V(\mathrm{Nrp}_\sigma))\setminus \Delta)/S_2$. Define a map $W\rightarrow \mathbf{Gr}(2,\mathrm{Symd}(A,\sigma))$ by sending a pair $(x,y)$ to the plane spanned by both of these points. One can check, by extending scalars to an algebraic closure $k^{alg}$, that this map realizes $W$ as an open subvariety of the Grassmannian $\mathbf{Gr}(2,\mathrm{Symd}(A,\sigma))$. Now $\acute{e}t_2(A)$ is open in $W$, and therefore also open in this Grassmannian. Since the Grassmannian variety is $R$-trivial, this implies that $\acute{e}t_2(A)$ is as well.
\end{proof}

The variety $\acute{e}t_2(A)$ inside $\mathrm{Skew}(X)$ is exactly the open subscheme whose points correspond to subschemes of $X$ which are geometrically a pair of skew lines. From Albert's theorem \cite[\S XI.6 Theorem 9]{MR0000595}, the set of $k$-rational points of $\acute{e}t_2(A)(k)$ is nonempty and from \cite[Proposition 6.2]{MR2601007} this implies that $\acute{e}t_2(A)$ is unirational. As in the proof of Theorem \ref{thm: kar}, the above Proposition \ref{prop: rtriv} implies that all of those subschemes of $X$ which correspond to fibers of these $k$-points of $\acute{e}t_2(A)$ represent the same cycle class in $\CH_1(X)$. Hence we've proved:

\begin{cor}\label{cor: curvescd}
Suppose that $A=Q_1\otimes Q_2$ is a biquaternion algebra. Let $F/k$ be a biquadratic Galois splitting field for $A$ with Galois group $G=\mathrm{Gal}(F/k)$. Choose generators $\sigma$ and $\tau$ so that $G=\{1,\sigma,\tau,\sigma\tau\}$. Pick a closed point $x\in X$ with residue field $k(x)=F$ and identify the $F$-points of $x_F$ compatibly with elements of $G$. 

Let $L_{\sigma,\sigma\tau}\subset X_F$ denote the line passing through $\sigma$ and $\sigma\tau$ and define the lines $L_{1,\sigma}$, $L_{\tau,\sigma\tau}$, and $L_{1,\tau}$ similarly. Let $C,D\subset X$ be the curves, which exist by descent, so that $C\times_k F=C'$ and $D\times_k F=D'$ where \[ C'=L_{1,\sigma}\cup L_{\tau,\sigma\tau} \quad \mbox{and}\quad D'=L_{1,\tau}\cup L_{\sigma,\sigma\tau}.\] Then there is an equality \[[C]=3c_2(\zeta_X(1))-c_1(\zeta_X(1))^2=[D]\] inside $\CH_1(X)$.
\end{cor}

\begin{proof}
We proved in Lemma \ref{lem: spec1} that there is a subscheme $V$ defining a point of $\acute{e}t_2(A)$ with $[V]=3c_2(\zeta_X(1))-c_1(\zeta_X(1))^2$. Both $C$ and $D$ also define $k$-points of $\acute{e}t_2(A)$, hence $[C]=[V]=[D]$.  
\end{proof}

\bibliographystyle{amsalpha}
\bibliography{bib}
\end{document}